\documentclass{article}

\usepackage{epsfig}
 \usepackage{epstopdf}
\usepackage[T1]{fontenc}
\usepackage{geometry}
\usepackage{amsbsy,amsmath,latexsym,amsfonts, epsfig, color, authblk, amssymb, graphics, bm}
\usepackage{epsf,slidesec,epic,eepic}
\usepackage{fancybox}
\usepackage{fancyhdr}
\usepackage{setspace}
\usepackage{nccmath}
\usepackage{color}
\usepackage[colorlinks,linkcolor=black,anchorcolor=black,citecolor=black,hyperindex,CJKbookmarks]{hyperref}
\newenvironment{proof}{{\noindent \bf Proof.}}{\hfill$\Box$\medskip}

\newtheorem{theorem}{Theorem}[section]
\newtheorem{corollary}[theorem]{Corollary}
\newtheorem{lemma}[theorem]{Lemma}

\newtheorem{definition}{Definition}[section]

\newtheorem{remark}{Remark}

\def \b{\beta}

\def \N{\mathbb{N}}

\def \A{\mathcal{A}}
\def \B{\mathcal{B}}
\def \D{\mathcal{D}}
\def \G{\mathcal{G}}
\def \H{\mathcal{H}}

\def \M{\mathcal{M}}

\begin{document}

\title{ Uniform recurrence properties for beta-transformation
}
\author{  Lixuan Zheng and Min Wu,\\
\small  \it  LAMA, Universit\'{e} Paris-Est Cr\'{e}teil, 61 av G\'{e}n\'{e}ral de Gaulle, 94010, Cr\'{e}teil, France\\
\it Department of Mathematics, South China University of Technology, Guangzhou 510640, P.R. China\\
\small \it E-mail: lixuan.zheng@u-pec.fr, wumin@scut.edu.cn}
\date{4th June,2018}
\date{}
\maketitle
\begin{center}
\begin{minipage}{120mm}{\small {\bf Abstract.} For any $\beta > 1$, let $T_\beta: [0,1)\rightarrow [0,1)$ be the $\beta$-transformation defined by $T_\beta x=\beta x \mod 1$. We study the uniform recurrence properties of the orbit of a point  under the $\beta$-transformation  to the point itself. The size of the set of points with prescribed uniform recurrence rate is obtained. More precisely, for any $0\leq \hat{r}\leq +\infty$, the set $$\left\{x \in [0,1): \forall\ N\gg1, \exists\ 1\leq n \leq N, {\rm\ s.t.}\ |T^n_\beta x-x|\leq \beta^{-\hat{r}N}\right\}$$ is of  Hausdorff dimension $\left(\frac{1-\hat{r}}{1+\hat{r}}\right)^2$ if $0\leq \hat{r}\leq 1$ and is countable if $\hat{r}>1$.}
\end{minipage}
\end{center}

\vskip0.5cm

\section{Introduction}
Let $(X,\B,\mu,T)$ be a measure-preserving dynamical system with a finite Borel measure $\mu$. Let $d$ be a metric on $X$. The well-known Poincar\'{e} Recurrence Theorem shows that typically the orbit of a point asymptotically approaches to the point itself. More precisely, $$\liminf_{n\rightarrow\infty}d(T^n x,x)=0$$ for $\mu$-almost all $x\in X$. Boshernitzan \cite{B} described the speed of such asymptotic recurrence. In fact, he proved that if there is some $\alpha>0$ such that the $\alpha$-dimensional Hausdorff measure $\H^{\alpha}$ is $\sigma$-finite on $X$ (i.e., $X$ can be written as a countable union of subsets $X_i$ with $\H^\alpha(X_i)<\infty$ for all $i=1,2,\ldots$), then$$\liminf_{n\rightarrow\infty}n^{\frac{1}{\alpha}}d(T^n x,x)<\infty$$for $\mu$-almost all $x\in X$, and that if $\H^{\alpha}(X)=0$, then $$\liminf_{n\rightarrow\infty}n^{\frac{1}{\alpha}}d(T^n x,x)=0$$ for $\mu$-almost all $x\in X$.  There are also many other studies on the asymptotic behavior of the orbits motivated by Poincar\'{e} Recurrence Theorem including the first return time \cite{BS}, dynamical Borel-Cantelli Lemma \cite{CK}, waiting time \cite{G}, shrinking target problems \cite{HV1,HV2,SW} and so on.

Different to the asymptotic way of approximation, the famous Dirichlet Theorem provides another point of view of the study on the approximation of the orbits: a uniform way. The Dirichlet Theorem  states that for any positive irrational real number $\theta$, for all real number $N\geq 1$, there is an integer $n$ with $1 \leq n \leq N$ satisfying
\begin{equation}\label{un}
\|n\theta\|<N^{-1},
\end{equation}where $\|\cdot\|$ denotes the distance to the nearest integer. The uniformity lies in that (\ref{un}) has an integer solution for any sufficiently large $N$. Note that $\|n\theta\|=\|T_\theta^n x-x\|$ where $T_\theta:\mathbb R/\mathbb Z\rightarrow \mathbb R/\mathbb Z$ is defined by $T_\theta x=x+\theta$. Thus, the Dirichlet Theorem can be explained as that under the dynamical system $(\mathbb R/\mathbb Z, T_\theta)$, all points $x$ uniformly return to the point itself with the speed $\frac{1}{N}$. Motivated by the Dirichlet Theorem, some results of the uniform approximation properties have already appeared in \cite{BM,BL,BLR,K,KL}.

In our paper, we want to investigate the uniform recurrence property of a point to itself  IN the beta-dynamical system. We aims at giving the sizes (Lebesgue measure and Huasdorff dimension) of the sets of points with prescribed uniform recurrence rate.

For any real number $\beta>1$, the \emph{$\beta$-transformation} $T_\beta$ on $[0,1)$ is defined by
\begin{equation}\label{t}
T_{\beta}(x) = \beta x\mod 1.
\end{equation} We consider the following two exponents of recurrence, one is for asymptotic recurrence, and the other is for uniform recurrence.

\begin{definition}
%Let $\beta>1$. For all $x\in[0,1)$, define $r_{\beta}(x)$ as the supremum of the real number $r$ for which the equation $$|T_{\beta}^nx-x|<(\beta^n)^{-r}$$ has infinitely many solutions $n\in\mathbb{N}$, and define $\hat{r}_{\beta}(x)$ as the supremum of the real number $\hat{r}$ such that for all large enough integer $N$, the equation $$|T_{\beta}^nx-x|<(\beta^N)^{-\hat{r}}$$ has a solution $n\in\mathbb N$ with $1\leq n\leq N$.
Let $\beta>1$. For all $x\in[0,1)$, define$$r_{\beta}(x):=\sup \{0\leq r\leq +\infty: |T_{\beta}^nx-x|<(\beta^n)^{-r}\ {\rm for \ infinitely\ many}\ n\in\mathbb{N}\}$$ and $$\hat{r}_\beta(x):=\sup\{0\leq \hat{r}\leq +\infty:\ {\rm for\ all}\ N\gg1,\ {\rm there\ is\ }n\in[1,N],\ {\rm s.t.\ }|T_{\beta}^nx-x|<(\beta^N)^{-\hat{r}}\}.$$
\end{definition}

The exponents $r_{\beta}(x)$  and $\hat{r}_\beta(x)$ are analogous to the exponents introduced in \cite{MB}, see also \cite{BM,BL}.  By the definitions of $r_{\beta}(x)$  and $\hat{r}_\beta(x)$, it can be checked that $\hat{r}_{\beta}(x)\leq r_{\beta}(x)$ for all $x\in[0,1)$. Actually, applying Philipp's result \cite{PH}, we can deduce that the set $\{x: r_\beta(x)=0\}$ is of full Lebesgue measure (see Section 3.1). The asymptotic exponent $r_\beta(x)$ has been studied by Tan and Wang \cite{TW} who showed that for all $0\leq r \leq +\infty$,
\begin{equation}\label{tw}
\dim_{\rm H}\left\{x \in [0,1):r_\beta(x)\geq r\right\}=\frac{1}{1+r},
\end{equation}where $\dim_{\rm H}$ denotes the Hausdorff dimension. We refer the readers to Falconer \cite{FE} for more properties of Hausdorrf dimension. Our main result is as follows.
\begin{theorem}\label{t1}
Let $\beta>1$. The set $\{x\in[0,1):\hat{r}_\beta(x)=0\}$ is of full Lebesgue measure. When $\hat{r}>1$, the set $\{x\in[0,1):\hat{r}_\beta(x)\geq\hat{r}\}$ is countable.  When $0\leq \hat{r}\leq 1$,  we have $$\dim_{\rm{H}}\{x\in[0,1):\hat{r}_\beta(x)\geq\hat{r}\}=\dim_{\rm{H}}\{x\in[0,1):\hat{r}_\beta(x)=\hat{r}\}=\left(\frac{1-\hat{r}}{1+\hat{r}}\right)^2.$$
\end{theorem}

Actually, Theorem \ref{t1} follows from the following more general result which gives the Hausdorff dimension of the set of points whose exponents $r_{\beta}(x)$ and $\hat{r}_\beta(x)$ are both prescribed. For all $0\leq r\le +\infty,\ 0\leq \hat{r}\leq +\infty$, let
$$R_\beta(\hat{r},r):=\left\{x \in [0,1):\hat{r}_\beta(x)=\hat{r},\ r_\beta(x)=r\right\}.$$
\begin{theorem}\label{t2}
Let $\beta>1$. The set $R_\beta(0,0)$ is of full Lebesgue measure.  When $0\leq \frac{r}{1+r}< \hat{r}\leq +\infty$, the set $R_{\beta}(\hat{r},r)$ is countable.  When $0\leq \hat{r} \leq \frac{r}{1+r},\ 0<r\leq+\infty$, we have $$\dim_{\rm H}R_{\beta}(\hat{r},r)=\frac{r-(1+r)\hat{r}}{(1+r)(r-\hat{r})}.$$
\end{theorem}

By Theorem \ref{t2}, the following new result related to the asymptotic exponent $r_\beta(x)$ is immediate.
\begin{corollary}Let $\beta>1$. For all $0\leq r\leq +\infty$, we have
$$\dim_{\rm H} \left\{x \in [0,1):r_\beta(x)= r\right\}=\frac{1}{1+r}.$$
\end{corollary}

Our paper is organized as follows. Section 2 contains a brief summary of some classical results on the $\beta$-transformation without proofs. Section 3 is devoted to the proof of Theorem \ref{t2}. The proof of Theorem \ref{t1} will be given in the last section.

\section{$\beta$-transformation}
In this section, we will provide some notations and properties on $\beta$-transformation. For more information on $\beta$-transformation, see \cite{BW,AB,P,R} and the references therein.

The $\beta$-transformation is first introduced by  R\'{e}nyi \cite{R}. By the iteration of $T_{\beta}$ defined by (\ref{t}), every real number $x \in [0, 1)$ can be uniquely expanded as:
$$x=\frac{\varepsilon_1(x,\beta)}{\b}+\cdots+\frac{\varepsilon_n(x,\beta)}{\beta^n}+\cdots,$$
where $\varepsilon_n(x,\beta)=\lfloor\beta T_{\beta}^{n-1}(x)\rfloor$ is the integer part of $\beta T_{\beta}^{n-1}(x)$ for all $n\in\N$. The integer $\varepsilon_n(x,\beta)$ is called the \emph{$n$-th digit of $x$}. We call the sequence $\varepsilon(x,\beta):=(\varepsilon_1(x,\beta), \ldots,\varepsilon_n(x,\beta),\ldots)$ the \emph{$\beta$-expansion of $x$}.
%R\'{e}nyi \cite{R} proved that the transformation $T_\beta$ has an invariant ergodic measure $\nu_\beta$ which is equivalent to the Lebesgue measure $\mathcal{L}$. More precisely, there exists a constant $c_\beta:=1-\frac{1}{\beta}$ such that for any Borel measurable set $A$, we have
%\begin{equation}\label{nu}
%c_\beta \nu_\beta(A)\leq \mathcal{L} (A) \leq \frac{1}{c_\beta}\nu_\beta(A).
%\end{equation}

We can see that every $n$-th digit $\varepsilon_n(x,\beta)$ belongs to $\A:=\{0,1,\cdots,\lceil\beta\rceil-1\}$ where $\lceil x\rceil$ means the smallest integer larger than $x$. A word $(\varepsilon_1,\ldots,\varepsilon_n)\in \A^n$ is said to be $\beta$-\emph{admissible} if there is a real number $x \in [0, 1)$ such that $(\varepsilon_1(x,\beta), \ldots,\varepsilon_n(x,\beta))=(\varepsilon_1,\ldots,\varepsilon_n)$. Similarly, an infinite sequence $(\varepsilon_1,\ldots,\varepsilon_n,\ldots)$ is called $\beta$-\emph{admissible} if there exists an $x \in [0, 1)$ such that $\varepsilon(x,\beta)=(\varepsilon_1,\ldots,\varepsilon_n,\ldots)$. Denote by $\Sigma_\b^n$ the family of all $\beta$-admissible words of length $n$, that is, $$\Sigma_\b^n=\{(\varepsilon_1,\ldots,\varepsilon_n)\in \A^n: \exists\ x \in [0,1),\ {\rm s.t.\ }\varepsilon_j(x,\b)=\varepsilon_j, \forall\ 1\leq j \leq n\}.$$  Write $\Sigma_\b^\ast=\bigcup\limits_{n=1}^\infty\Sigma_\b^n$ the set of all $\beta$-admissible words with finite length. Denote by $\Sigma_\b$ the set of all $\beta$-admissible sequences, that is, $$\Sigma_\b=\{(\varepsilon_1,\varepsilon_2,\ldots)\in \A^\N: \exists\ x \in [0,1),\ {\rm s.t.}\ \varepsilon(x,\beta)=(\varepsilon_1,\varepsilon_2,\ldots)\}.$$

The \emph{lexicographical order $<_{\rm{lex}}$} in the space $\A^\N$ is defined as follows: $$(\omega_1,\omega_2,\ldots)<_{\rm{lex}}(\omega'_1,\omega'_2,\ldots)$$if $\omega_1<\omega'_1$ or there exists $j > 1$, such that for all $1 \leq k\leq j-1$, we have $\omega_k=\omega'_k$  but $\omega_j<\omega'_j$. The symbol $\leq_{\rm{lex}}$ stands for $=$ or $<_{\rm{lex}}$.
%Moreover, for any $n,m \geq 1$, $(\omega_1,\ldots,\omega_n) <_{\rm{lex}}(\omega'_1,\ldots,\omega'_m)$ means $(\omega_1,\ldots,\omega_n,0^\infty) <_{\rm{lex}}(\omega'_1,\ldots,\omega'_m,0^\infty)$.

We now extend the definition of the $\beta$-transformation to $x=1$. Let $T_\beta(1)=\beta-\lfloor \beta\rfloor$. We have $$1=\frac{\varepsilon_1(1,\beta)}{\beta}+\cdots+\frac{\varepsilon_n(1,\beta)}{\beta^n}+\cdots,$$ where $\varepsilon_n(1,\beta)=\lfloor\beta T_\beta^n(1)\rfloor.$ Specially, if the $\beta$-expansion of $1$ is finite, that is, there is an integer $m\geq 1$ such that $\varepsilon_m(1,\beta)> 0$ and  $\varepsilon_k(1,\beta)=0$ for all $k> m$, $\beta$ is called a \emph{simple Parry number}. In this case, set $$\varepsilon^\ast(\beta):=(\varepsilon_1^\ast(\beta),\varepsilon_2^\ast(\beta),\ldots)=(\varepsilon_1(1,\beta),\varepsilon_2(1,\beta),\ldots, \varepsilon_m(1,\beta)-1)^\infty$$ where $\omega^\infty=(\omega,\omega,\ldots)$. If the $\beta$-expansion of $1$ is not finite, set $\varepsilon^\ast(\beta)=\varepsilon(1,\beta)$. In both cases, we have $$1=\frac{\varepsilon_1^\ast(\beta)}{\beta}+\cdots+\frac{\varepsilon_n^\ast(\beta)}{\beta^n}+\cdots.$$ The sequence $\varepsilon^\ast(\beta)$ is consequently called \emph{the infinite $\beta$-expansion of $1$}. For any $N$ with $\varepsilon_N^\ast(\beta)>0$, let $\beta_N>1$ be the unique solution of the equation $$1=\frac{\varepsilon_1^\ast(\beta)}{x}+\cdots+\frac{\varepsilon_N^\ast(\beta)}{x^N}.$$ Immediately, the infinite $\beta_N$-expansion of $1$ is $$\varepsilon^\ast(\beta_N)=(\varepsilon_1^\ast(\beta),\ldots,\varepsilon_{N}^\ast(\beta)-1)^\infty.$$ We have $0<\beta_N<\beta$ and $\beta_N\rightarrow \beta$ as $N\rightarrow +\infty$. The real number $\beta_N$ is therefore called an \emph{approximation of $\beta$}.

The following theorem is a characterization of the $\beta$-admissible words established by Parry \cite{P}. It indicates that the $\beta$-dynamical system is totally determined by the infinite $\beta$-expansion of $1$.

\begin{theorem}[Parry \cite{P}]\label{P} Let $\beta >1$.

(1) For any $n\in \N$, $\omega=(\omega_1,\ldots,\omega_n)\in \Sigma_\beta^n$, if and only if, $$(\omega_{j},\ldots,\omega_{n}) \leq_{\rm{lex}} (\varepsilon_1^\ast,\ldots,\varepsilon_{n-j}^\ast),\ \forall\ 1 \leq j \leq n.$$

(2) If $1<\beta_1<\beta_2$, then $\varepsilon^\ast(\beta_1)<_{\rm{lex}} \varepsilon^\ast(\beta_2)$. For every $n \geq 1$, it holds that $$\Sigma_{\beta_1}^n \subseteq \Sigma_{\beta_2}^n.$$
\end{theorem}

The following theorem due to R\'{e}nyi \cite{R} shows that the dynamical system $([0,1), T_\beta)$ has topological entropy  $\log_\beta$. Here and subsequently, we denote by $\sharp$ the cardinality of a finite set.
\begin{theorem}[R\'{e}nyi \cite{R}]\label{R}
For any $n \geq 1$, we have $$\beta^n \leq \sharp \Sigma_\beta^n \leq \frac{\beta^{n+1}}{\beta-1}{\rm\quad and}\quad \lim_{n\rightarrow \infty}\frac{\log\sharp \Sigma_\beta^n}{n}=\log \beta.$$
\end{theorem}

For any $\beta$-admissible word $\omega=(\omega_1,\ldots,\omega_n)$, let $$I_n(\omega):=I_n(\omega,\beta)= \{x \in [0,1): \varepsilon_j(x,\b)=\omega_j,\ \forall\  1 \leq j \leq n\}.$$ The set $I_n(\omega)$ is called the \emph{cylinder of order $n$ associated to the $\beta$-admissible word $\omega$}.  It can be checked that the cylinder $I_n(\omega)$ is a left-closed and right-open interval (see \cite{AB}). Denote by $|I_n(\omega)|$ the length of $I_n(\omega)$. Then $|I_n(\omega)|\leq \beta^{-n}$. Let $I_n(x,\beta)$ be the cylinder of order $n$ which contains $x\in [0,1)$. To shorten notation, we write $I_n(x)$ instead of $I_n(x,\beta)$ and denote by $|I_n(x)|$ its length. The cylinder of order $n$ is called \emph{full} if  $|I_n(\omega)| = \b^{-n}$. The corresponding word of the full cylinder is also said to be {\it full}.

The full word plays an important role in constructing a Cantor set for the aim of estimating the lower bound of $\dim_{\rm H}R_\beta(\hat{r},r)$. A characterization of full words was given by Fan and Wang \cite{AB} as follows.
\begin{theorem}[Fan and Wang \cite{AB}]\label{AB}
For any $n\in\N$, the word $\omega=(\omega_1,\ldots,\omega_n)$ is full if and only if for all $m \in \N$ and $\omega'= (\omega'_1,\ldots,\omega'_m)\in \Sigma_\b^m$, the concatenation $\omega \ast \omega'=(\omega_1, \ldots,\omega_n,\omega'_1,\ldots, \omega'_m)$ is still $\beta$-admissible. Moreover, if $(\omega_1,\ldots,\omega_{n-1},\omega'_n)$ with $\omega'_n > 0$ is $\beta$-admissible, then $(\omega_1,\ldots,\omega_{n-1},\omega_n)$ is full for any  $0 \leq \omega_n < \omega'_n$.

%(2) For all $n\in\N$, if $(\omega_1,\ldots,\omega_{n-1},\omega'_n)$ with $\omega'_n > 0$ is $\beta$-admissible, then $(\omega_1,\ldots,\omega_{n-1},\omega_n)$ is full for every  $0 \leq \omega_n < \omega'_n$.

%(2) For all $n\in\N$, if $\omega=(\omega_1,\ldots,\omega_n)$ is full, then we have $$|I_{n+m}(\omega_1,\ldots, \omega_n,\omega'_1,\ldots,\omega'_m)| = \beta^{-n}\cdot |I_m(\omega'_1,\ldots,\omega'_m)|,$$ for any $(\omega'_1,\ldots,\omega'_m)\in \Sigma_\b^m$.
\end{theorem}

It follows from Theorem \ref{P}(2) that $\Sigma_{\beta_N}^n\subseteq\Sigma_{\beta}^n$ for all $n\geq 1$. For any $\omega\in \Sigma_{\beta_N}^\ast$, by Theorem \ref{AB}, the word $(\omega,0^N)$ is full when $\omega$ is regarded as an element of $\Sigma_{\beta}^n$. As a result,
\begin{equation}\label{in}
\beta^{-(n+N)}\leq |I_n(\omega,\beta)|\leq \beta^{-n}.
\end{equation}

Furthermore, the following theorem due to Bugeaud and Wang \cite{BW} implies that the full cylinders are well distributed in the unit interval $[0,1)$.
\begin{theorem}[Bugeaud and Wang \cite{BW}]\label{no}
There is at least one full cylinder for all $n+1$ consecutive cylinders of order $n$.
\end{theorem}

We end this section by giving the modified mass distribution principle shown by Bugeaud and Wang \cite{BW} which is  an important tool to estimate the lower bound of the Hausdorff dimension in $\beta$-dynamical system.  For convenience, for all $x\in[0,1)$, we write $I_n(x)$ as $I_n$ without any ambiguity.
\begin{theorem}[Bugeaud and Wang \cite{BW}]\label{mp}
Let $\mu$ be a Borel measure supported on $E$. If there exists a constant $c>0$ such that for all integer $n$ large enough , the inequality $\mu(I_n)\leq c|I_n|^s$ holds for all cylinders $I_n$. Then $\dim_{\rm H}E\geq s.$
\end{theorem}

\section{ Proof of Theorems \ref{t2}}
Before our proof, we will give some useful lemmas.
\begin{lemma}\label{mn}
For any $x\in [0,1)$ whose $\beta$-expansion is not periodic and $r_\beta(x)>0$, there exist two sequences   $\{n_{k}\}_{k=1}^\infty$ and $\{m_{k}\}_{k=1}^\infty$ such that
\begin{equation}\label{sup}
r_\beta(x)=\limsup_{k\rightarrow \infty}\frac{m_{k}-n_{k}}{n_{k}}
\end{equation}
and
\begin{equation}\label{inf}
\hat{r}_\beta(x)=\liminf_{k\rightarrow \infty}\frac{m_{k}-n_{k}}{n_{{k}+1}}.
\end{equation}
\end{lemma}
\begin{proof}
Suppose $\varepsilon(x,\beta)=(\varepsilon_1,\varepsilon_2,\ldots)$. Let $$n_1=\min\{n\geq1:\varepsilon_{n+1}=\varepsilon_1\},\ \ m_1=\max\{n\geq n_1: |T_\beta^{n_1} x-x|< \beta^{-(n-n_1)}\}.$$ Suppose that for all $k\geq 1$, $n_k$ and $m_k$ have been defined. Set $$n_{k+1}=\min\{n\geq n_k:\varepsilon_{n+1}=\varepsilon_1\},\  m_{k+1}=\max\{n\geq n_{k+1}: |T_\beta^{n_k} x-x|< \beta^{-(n-n_k)}\}.$$ Note that $r_\beta(x)>0$. We always can find the position $n_k$ such that $\varepsilon_{n_k+1}$ returns to $\varepsilon_1$ which means $n_k$ is well defined. Since $\beta^{-n}$ is decreasing to $0$ as $n$ goes to infinity and $\varepsilon(x,\beta)$ is not periodic, $m_k$ is well defined. By the definitions of $n_k$ and $m_k$,  for all $k\geq1$, we have $$\beta^{-(m_k-n_k)-1}\leq |T_\beta^{n_k} x-x|< \beta^{-(m_k-n_k)}.$$

Now we choose two subsequences $\{n_{i_k}\}_{k=1}^\infty$ and $\{m_{i_k}\}_{k=1}^\infty$  of $\{n_k\}_{k=1}^\infty$ and $\{m_k\}_{k=1}^\infty$ such that $\{m_{i_k}-n_{i_k}\}_{k=1}^\infty$ is not decreasing. Let $i_1=1$. Assume that $i_k$ has been defined. Let $$i_{k+1}=\min\{i>i_k:m_i-n_i>m_{i_k}-n_{i_k}\}.$$Since $r_\beta(x)>0$, it follows that $m_k-n_k$ goes to infinity as $k\rightarrow+\infty$. So $i_{k+1}$ is well defined. Then we have the sequence $\{m_{i_k}-n_{i_k}\}_{k=1}^\infty$ is not decreasing. Without causing any confusion, we still use the same symbols $\{n_k\}_{k=1}^\infty$ and $\{m_k\}_{k=1}^\infty$ to substitute the subsequences $\{n_{i_k}\}_{k=1}^\infty$ and $\{m_{i_k}\}_{k=1}^\infty$.   We claim that
$$r_\beta(x)=\limsup_{k\rightarrow \infty}\frac{m_{k}-n_{k}}{n_{k}}\quad {\rm and}\quad
\hat{r}_\beta(x)=\liminf_{k\rightarrow \infty}\frac{m_{k}-n_{k}}{n_{{k}+1}}.$$
In fact, assume $$\limsup\limits_{k\rightarrow \infty}\frac{m_{k}-n_{k}}{n_{k}}=c.$$ On the one hand, there is a subsequence $\{j_k\}_{k=1}^\infty$ such that$$\lim_{k\rightarrow \infty}\frac{m_{j_k}-n_{j_k}}{n_{j_k}}=c.$$ This implies that for all $\delta>0$, there is an integer $k_0$, for each $k\geq k_0$, $m_{j_k}-n_{j_k}\geq (c-\delta)n_{j_k}$. So $$|T_\beta^{n_{j_k}}x-x|< \beta^{n_{j_k}-m_{j_k}}\leq \beta^{-(c-\delta)n_{j_k}}.$$ Consequently, $r_\beta(x)\geq c-\delta$ for all $\delta\geq 0$. On the other hand, there is an integer $k_0$, for any $k\geq k_0$, we have $m_{k}-n_{k}\leq (c+\delta)n_k$. So for all $n\geq n_{k_0}$, there is an integer $k$ such that $n_k\leq n<n_{k+1}$. This means $$|T_\beta^nx-x|\geq\beta^{-(m_k-n_k)-1}\geq \beta^{-(c+\delta)n_k}.$$ Hence, $r_\beta(x)<c+\delta$ for any $\delta>0$. Immediately, $r_\beta(x)= c$. The same argument can deduce the equality $(\ref{inf})$, we leave it to the readers.
\end{proof}

\begin{lemma}
For all $x\in [0,1)$ whose $\beta$-expansion is not periodic and $r_\beta(x)>0$, let $\{m_k\}_{k=1}^\infty$ and $\{n_k\}_{k=1}^\infty$ be defined in Lemma \ref{mn}. Then there is a sequence $\{t_k\}_{k=1}^\infty$ such that when $t_k\geq m_k$, we have
\begin{equation}\label{e4}
\left(\varepsilon_1,\ldots,\varepsilon_{m_{k}}\right)=\left(\varepsilon_1,\ldots,\varepsilon_{n_{k}},\varepsilon_1,\ldots, \varepsilon_{m_{k}-n_{k}}\right).
\end{equation}When $t_k<m_k$, we have \begin{equation}\label{1}
(\varepsilon_1,\ldots,\varepsilon_{m_k})=\left(\varepsilon_1,\ldots,\varepsilon_{n_k},\varepsilon_1,\ldots,\varepsilon_{t_k-n_k}, \varepsilon_{t_k-n_k+1}-1, \varepsilon_1^\ast(\beta), \ldots, \varepsilon_{m_k-t_k-1}^\ast(\beta)\right).
\end{equation} or \begin{equation}\label{2}
(\varepsilon_1,\ldots,\varepsilon_{m_k})=(\varepsilon_1,\ldots,\varepsilon_{n_k},\varepsilon_1,\ldots,\varepsilon_{t_k},\varepsilon_{t_k-n_k+1}+1, 0^{m_k-n_k-t_k-1}).
\end{equation}
\end{lemma}
\begin{proof}
Let
\begin{equation}\label{tk}
t_k=\max\{n>n_k:(\varepsilon_{n_k+1}\ldots,\varepsilon_n) =(\varepsilon_1,\ldots,\varepsilon_{n-n_k})\}.
\end{equation} That is, the position $t_k$ is chosen to make sure that $(\varepsilon_{n_k+1},\ldots,\varepsilon_{t_k})$ is the maximal block after position $n_k$ which returns to  $(\varepsilon_1,\ldots,\varepsilon_{t_k-n_k})$. Since $\varepsilon(x,\beta)$ is not periodic, $t_k$ is well defined. By the definition of $t_k$, it holds that
\begin{equation}\label{e1}
\left(\varepsilon_1,\ldots,\varepsilon_{t_{k}}\right)=\left(\varepsilon_1,\ldots,\varepsilon_{n_{k}},\varepsilon_1,\ldots, \varepsilon_{t_{k}-n_{k}}\right),\ \varepsilon_{t_{k}+1}\neq \varepsilon_{t_{k}-n_{k}+1}.
\end{equation}
Then $T_\beta^{n_k}x$ and $x$ belong to the interval $I_{t_k-n_k}(\varepsilon_1,\ldots,\varepsilon_{t_k-n_k})$ which implies $$\beta^{-(m_k-n_k)-1}\leq |T_\beta^{n_k} x-x|\leq |I_{t_k-n_k}(\varepsilon_1,\ldots,\varepsilon_{t_k-n_k})|\leq \beta^{-(t_k-n_k)}.$$ Thus, $t_k\leq m_k+1$.

When $t_k\geq m_k$, it follows from (\ref{e1}) that
$$\left(\varepsilon_1,\ldots,\varepsilon_{m_{k}}\right)=\left(\varepsilon_1,\ldots,\varepsilon_{n_{k}},\varepsilon_1,\ldots, \varepsilon_{m_{k}-n_{k}}\right).$$

When $t_k<m_k$ and $|\varepsilon_{t_{k}+1}- \varepsilon_{t_{k}-n_{k}+1}|\geq 2,$  by (\ref{e1}), there is a full cylinder of order $t_k-n_k+1$ between $x$ and $T_\beta^{n_k}$. As a consequence, $$|T_\beta^{n_k}x-x|\geq \beta^{-(t_k-n_k+1)}\geq \beta^{-(m_k-n_k)},$$ which contradicts with the definition of $m_k$. Hence, $|\varepsilon_{t_{k}+1}- \varepsilon_{t_{k}-n_{k}+1}|=1$.

When $\varepsilon_{t_{k}-n_{k}+1}=\varepsilon_{t_{k}+1}+1$. By the definition of $t_k$, we have $$x=\frac{\varepsilon_1}{\beta}+\cdots+\frac{\varepsilon_{t_k-n_k}}{\beta^{t_k-n_k}}+\frac{\varepsilon_{t_k-n_k+1}}{\beta^{t_k-n_k+1}}+\cdots\geq \frac{\varepsilon_1}{\beta}+\cdots+\frac{\varepsilon_{t_k-n_k}}{\beta^{t_k-n_k}}+\frac{\varepsilon_{t_k+1}+1}{\beta^{t_k-n_k+1}}. $$ Noting that $|T^{n_k}_\beta x-x|=x-T^{n_k}_\beta x< \beta^{-(m_k-n_k)}$, we have
$$T^{n_k}_\beta x> x-\beta^{-(m_k-n_k)}\geq\frac{\varepsilon_1}{\beta} +\cdots+\frac{\varepsilon_{t_k}}{\beta^{t_k-n_k}}+\frac{\varepsilon_{t_k+1}+1}{\beta^{t_k-n_k+1}}-\frac{1}{\beta^{m_k-n_k}}$$
Note that $$\frac{1}{\beta^{t_k-n_k+1}} -\frac{1}{\beta^{m_k-n_k}}=\frac{\varepsilon_{t_k+1}}{\beta^{t_k-n_k+1}}+\frac{\varepsilon_1^\ast(\beta)}{\beta^{t_k-n_k+2}} +\cdots+\frac{\varepsilon_{m_k-n_k-t_k-1}^\ast(\beta)-1}{\beta^{m_k-n_k}}+\cdots.$$ Thus,$$T^{n_k}_\beta x> \frac{\varepsilon_1}{\beta} +\cdots+\frac{\varepsilon_{t_k-n_k}}{\beta^{t_k-n_k}}+\frac{\varepsilon_{t_k+1}}{\beta^{t_k-n_k+1}}+\frac{\varepsilon_1^\ast(\beta)}{\beta^{t_k-n_k+2}} +\cdots+\frac{\varepsilon_{m_k-n_k-t_k-1}^\ast(\beta)-1}{\beta^{m_k-n_k}}+\cdots.$$Then,
$$(\varepsilon_1,\ldots,\varepsilon_{m_k})=\left(\varepsilon_1,\ldots,\varepsilon_{n_k},\varepsilon_1,\ldots,\varepsilon_{t_k-n_k}, \varepsilon_{t_k-n_k+1}-1, \varepsilon_1^\ast(\beta), \ldots, \varepsilon_{m_k-t_k-1}^\ast(\beta)\right).$$

When $\varepsilon_{t_k+1}=\varepsilon_{t_k-n_k+1}+1$,  we have $$x=\frac{\varepsilon_1}{\beta}+\cdots+\frac{\varepsilon_{t_k-n_k}}{\beta^{t_k-n_k}}+\frac{\varepsilon_{t_k-n_k+1}}{\beta^{t_k-n_k+1}}+\cdots\leq \frac{\varepsilon_1}{\beta}+\cdots+\frac{\varepsilon_{t_k-n_k}}{\beta^{t_k-n_k}}+\frac{\varepsilon_{t_k+1}}{\beta^{t_k-n_k+1}}. $$ Noticing that $|T^{n_k}_\beta x-x|=T^{n_k}_\beta x-x< \beta^{-(m_k-n_k)}$, we have
$$ T^{n_k}_\beta < x+\beta^{-(m_k-n_k)} \leq\frac{\varepsilon_1}{\beta} +\cdots+\frac{\varepsilon_{t_k-n_k}}{\beta^{t_k-n_k}}+\frac{\varepsilon_{t_k+1}}{\beta^{t_k-n_k+1}}+\frac{1}{\beta^{m_k-n_k}}.$$ This implies
$$(\varepsilon_1,\ldots,\varepsilon_{m_k})=(\varepsilon_1,\ldots,\varepsilon_{n_k},\varepsilon_1,\ldots,\varepsilon_{t_k},\varepsilon_{t_k-n_k+1}+1, 0^{m_k-n_k-t_k-1}).$$
\end{proof}
\begin{remark}
The two sequences $\{n_k\}_{k=1}^\infty$ and $\{m_k\}_{k=1}^\infty$ chosen here are exactly the same sequences in Bugeaud and Liao \cite{BL}. However, the property that $n_{k+1}>m_k=t_k>n_k$ which always holds in  \cite{BL} fails in our case. Then uncertainty of the relationship between $m_k$, $t_k$ and $n_{k+1}$ makes the construction of $\varepsilon(x,\beta)$ of all $x\in R_\beta(r,\hat{r})$ be much more complicated.
%Meanwhile, the estimation of the upper bound of $\dim_{\rm H}R_\beta(r,\hat{r})$  for the case $0\leq \hat{r}\leq \frac{r}{1+r},\ 0<r\leq +\infty$ also becomes much more difficult.
\end{remark}

Now we will investigate the relationship between $m_k$ and $n_k$ for the case $0\leq \frac{r}{1+r}<\hat{r}\leq +\infty$.
\begin{lemma}\label{com}
For each $x\in R_\beta(\hat{r},r)$ with $0\leq \frac{r}{1+r}<\hat{r}\leq +\infty$, let $m_k$ and $n_k$ be defined in Lemma \ref{mn}. Then  there is an integer $k'$ such that $n_{k+1}< m_k$ for all $k\geq k'$.
\end{lemma}
\begin{proof}
Since $\hat{r}>\frac{r}{1+r}$, by (\ref{sup}) and (\ref{inf}), there is a positive number $\varepsilon>0$ and an integer $k''$, such that for any $k\geq k''$, we have $$\frac{m_{k}-n_{k}}{n_{{k}+1}}\geq \left(\frac{r}{1+r}+\varepsilon\right)\frac{1}{1-\varepsilon}.$$
By (\ref{sup}), we obtain $$\frac{r}{1+r}=1-\frac{1}{1+r}=1-\frac{1}{1+\limsup\limits_{k\rightarrow \infty}\frac{m_{k}-n_{k}}{n_{k}}}=\limsup_{k\rightarrow \infty}\frac{m_{k}-n_{k}}{m_{k}}.$$ Consequently, there is an integer $k'\geq k''$, such that for all $k\geq k'$, $$\frac{m_{k}-n_{k}}{n_{{k}+1}}\geq \left(\frac{r}{1+r}+\varepsilon\right)\frac{1}{1-\varepsilon}\geq\frac{1}{1-\varepsilon}\cdot \frac{m_{k}-n_{k}}{m_{k}},$$ which implies $n_{k+1}\leq(1-\varepsilon)m_{k}.$
\end{proof}

The following lemma gives the sequences we will use to construct the covering of $R_\beta(r,\hat{r})$ when proving the upper bound of $\dim_{\rm H}R_\beta(r,\hat{r})$  for the case $0\leq \hat{r}\leq \frac{r}{1+r},\ 0<r\leq +\infty$.
\begin{lemma}\label{rep}
For all $x\in R_\beta(\hat{r},r)$ with $0\leq \hat{r}\leq 1,\ \frac{\hat{r}}{1-\hat{r}}\leq r<+\infty$, there are two sequences $\{n_k'\}_{k=1}^\infty$ and $\{m_k'\}_{k=1}^\infty$ such that for any large enough $k$, there is a positive real number $C$ satisfying
$$n'_{k+1}\geq (1-\delta)m'_k,\quad n'_{k+1}\geq \frac{2+\hat{r}_\beta(x)}{\hat{r}_\beta(x)}n'_k\quad {\rm and}\quad k\leq C\log n'_k.$$
\end{lemma}
\begin{proof}
For all $x\in R_\beta(\hat{r},r)\left(0\leq \hat{r}\leq 1,\ \frac{\hat{r}}{1-\hat{r}}\leq r<+\infty\right)$, the $\beta$-expansion of $x$ is not periodic and $r_\beta(x)>0$. Let $\{n_k\}_{k=1}^\infty$ and $\{m_k\}_{k=1}^\infty$ be the sequences defined in Lemma \ref{mn}.

By (\ref{inf}), for any $0<\delta_1<\frac{\hat{r}_\beta(x)}{2}$, for sufficiently large $k$, we have
\begin{equation}\label{in1}
m_k-n_k\geq (\hat{r}_\beta(x)-\delta_1)n_{k+1},\ {\rm that\ is,\ } m_k\geq n_{k+1}+(\hat{r}_\beta(x)-\delta_1)n_k\geq(1+\hat{r}_\beta(x)-\delta_1)n_k.
\end{equation}Note that $\hat{r}_\beta(x) \leq \frac{r_\beta(x)}{1+r_\beta(x)}$. It follows from (\ref{sup}) and (\ref{inf}) that $$\liminf_{k\rightarrow \infty}\frac{m_k-n_k}{n_{k+1}}\leq \limsup_{k\rightarrow \infty}\frac{m_k-n_k}{m_k}.$$
Consequently, $$\liminf_{k\rightarrow \infty}\frac{m_k-n_k}{n_{k+1}}- \limsup_{k\rightarrow \infty}\frac{m_k-n_k}{m_k}=\liminf_{k\rightarrow \infty} \left(\frac{m_k-n_k}{n_{k+1}}-\frac{m_k-n_k}{m_k}\right)\leq 0.$$

We claim that, for all $\delta>0$, there are infinitely many $k$, such that
\begin{equation}\label{seq}
n_{k+1}\geq (1-\delta)m_k\quad {\rm and}\quad n_{k+1}\geq \left(1+\frac{\hat{r}_\beta(x)}{2}\right)n_k=\frac{2+\hat{r}_\beta(x)}{2}n_k.
\end{equation}

In fact, notice that $m_k-n_k>0,\ n_{k+1}m_k>0$ for all $k\geq 1$. When $$\liminf_{k\rightarrow \infty}\left(\frac{m_k-n_k}{n_{k+1}}-\frac{m_k-n_k}{m_k}\right)=\liminf_{k\rightarrow \infty}\frac{(m_k-n_k)(m_k-n_{k+1})}{n_{k+1}m_k}<0,$$  by (\ref{in1}), we have $$n_{k+1}>m_k\geq(1+\hat{r}_\beta(x)-\delta_1)n_k\geq \left(1+\frac{\hat{r}_\beta(x)}{2}\right)n_k.$$ When $$\liminf_{k\rightarrow \infty}\left(\frac{m_k-n_k}{n_{k+1}}-\frac{m_k-n_k}{m_k}\right)=\liminf_{k\rightarrow \infty}\frac{(m_k-n_k)(m_k-n_{k+1})}{n_{k+1}m_k}=0,$$ then for all $0<\delta_2\leq\frac{(\hat{r}_\beta(x)-\delta_1)(\hat{r}_\beta(x)-2\delta_1)}{2(1+\hat{r}_\beta(x)-\delta_1)},$ there are infinitely many $k\geq k_0$ such that $$m_k-n_{k+1}\leq \delta_2\cdot\frac{n_{k+1}m_k}{m_k-n_k}\leq \delta_2\cdot \frac{n_{k+1}m_k}{(\hat{r}_\beta(x)-\delta_1)n_{k+1}} =\frac{\delta_2}{\hat{r}_\beta(x)-\delta_1}\cdot m_k,$$where the last inequality follows from the first part of (\ref{in1}).

Therefore, by (\ref{in1}) and the choice of $\delta_2$,
$$n_{k+1}\geq\left(1-\frac{\delta_2}{\hat{r}_\beta(x)-\delta_1}\right)m_k\geq \left(1-\frac{\delta_2}{\hat{r}_\beta(x)-\delta_1}\right)(1+\hat{r}_\beta(x)-\delta_1)n_k\geq \left(1+\frac{\hat{r}_\beta(x)}{2}\right)n_k.$$Letting $\delta=\frac{\delta_2}{\hat{r}_\beta(x)-\delta_1}$, we obtain the claim.

Now we choose the subsequence $\{n_{j_k}\}_{k=1}^\infty$ and $\{m_{j_k}\}_{k=1}^\infty$ of $\{n_{k}\}_{k=1}^\infty$ and $\{m_{k}\}_{k=1}^\infty$ satisfying (\ref{seq}). For simplicity, let $n'_k=n_{j_k}$ and $m'_k=m_{j_k}$. Then $\{n'_k\}_{k=1}^\infty$ increases at least exponentially since $$n'_{k+1}\geq \frac{2+\hat{r}_\beta(x)}{\hat{r}_\beta(x)}n'_k.$$ In conclusion, the new sequences satisfy that,  there is a large enough $k_0$ and a positive real number $C$ such that for all $k\geq k_0$, we have $k\leq C\log n'_k.$
\end{proof}

We now give an estimation of the numbers of the sum of all the lengths of the blocks which are ``fixed'' in the prefix of length $m_k$ ($m_k$ is defined in Lemma \ref{rep}) of the infinite sequence $\varepsilon(x,\beta)$ where  $x\in R_\beta(\hat{r},r)\ \left(0\leq \hat{r}\leq 1,\ \frac{\hat{r}}{1-\hat{r}}\leq r<+\infty\right)$ for all sufficiently large $k$.
\begin{lemma}\label{pre}
For all $x\in R_\beta(\hat{r},r)$ with $0\leq \hat{r}\leq 1,\ \frac{\hat{r}}{1-\hat{r}}\leq r<+\infty$, let $\{m'_k\}_{k=1}^\infty$ and $\{n'_k\}_{k=1}^\infty$ be the sequences defined in Lemma \ref{rep}. Then for all large enough integer $k$, we have
\begin{equation}\label{i3}
\sum_{i=1}^k (m'_i-n'_i)\geq n'_{k+1}\left(\frac{\hat{r}_\beta(x)r_\beta(x)}{r_\beta(x)-\hat{r}_\beta(x)}-\epsilon'\right)
\end{equation}
\end{lemma}
\begin{proof}
By Lemma \ref{rep}, the sequences $\{n'_k\}_{k=1}^\infty$ and $\{m'_k\}_{k=1}^\infty$  are the subsequences of$\{n_k\}_{k=1}^\infty$ and $\{m_k\}_{k=1}^\infty$, the equalities (\ref{sup}) and (\ref{inf}) become:
\begin{equation}\label{sup1}
r_\beta(x)\geq\limsup_{k\rightarrow \infty}\frac{m'_k-n'_k}{n'_k}
\end{equation}
and
\begin{equation}\label{inf1}
\hat{r}_\beta(x)\leq\liminf_{k\rightarrow \infty}\frac{m'_k-n'_k}{n'_{k+1}}.
\end{equation}
By (\ref{sup1}) and (\ref{inf1}), for any real number $0<\epsilon<\frac{\hat{r}_\beta(x)}{2}$, there exits a large enough $k_1$, such that for all $k\geq k_1$,
\begin{equation}\label{i1}
m'_k-n'_k\leq (r_\beta(x)+\epsilon)n'_k
\end{equation}
and
\begin{equation}\label{i2}
m'_k-n'_k \geq (\hat{r}_\beta(x)-\epsilon)n'_{k+1}.
\end{equation}
By (\ref{i1}) and (\ref{i2}), we have $$n'_k\geq\frac{m'_k-n'_k}{r_\beta(x)+\epsilon} \geq \frac{\hat{r}_\beta(x)-\epsilon}{r_\beta(x)+\epsilon}n'_{k+1}.$$ Hence,
\begin{equation}\label{sum}
\sum_{i=1}^k (m'_i-n'_i)\geq\sum_{i=k_1}^k (\hat{r}_\beta(x)-\epsilon)n'_{i+1} \geq (\hat{r}_\beta(x)-\epsilon)n'_{k+1}\sum_{i=0}^{k-k_1}\left(\frac{\hat{r}_\beta(x)-\epsilon}{r_\beta(x)+\epsilon}\right)^i.
\end{equation}Let $$n'=(\hat{r}_\beta(x)-\epsilon)n'_{k+1}\sum_{i=k-k_1+1}^{\infty}\left(\frac{\hat{r}_\beta(x)-\epsilon}{r_\beta(x)+\epsilon}\right)^i.$$ Note that $\frac{\hat{r}_\beta(x)-\epsilon}{r_\beta(x)+\epsilon}<1$. There exists $k'\geq k$ such that for all $\ell\geq k'$, we have $$n'_{k+1}\left(\frac{\hat{r}_\beta(x)-\epsilon}{r_\beta(x)+\epsilon}\right)^\ell< \frac{1}{\ell^2}.$$ Thus, $$n'_{k+1}\sum_{i=k'+1}^\infty\left(\frac{\hat{r}_\beta(x)-\epsilon}{r_\beta(x)+\epsilon}\right)^i<\sum_{i=k'+1}^\infty \frac{1}{i^2}<+\infty.$$ This implies  $$n'=(\hat{r}_\beta(x)-\epsilon)n'_{k+1}\sum_{i=k-k_1+1}^{k'}\left(\frac{\hat{r}_\beta(x)-\epsilon}{r_\beta(x)+\epsilon}\right)^i+ (\hat{r}_\beta(x)-\epsilon)n'_{k+1}\sum_{i=k'+1}^{\infty}\left(\frac{\hat{r}_\beta(x)-\epsilon}{r_\beta(x)+\epsilon}\right)^i<+\infty.$$
By (\ref{sum}), for any small enough real number $\epsilon'>0$, there is a sufficiently large integer $k\geq\{k',k_1\}$ such that
$$\begin{aligned}\sum_{i=1}^k (m'_i-n'_i)&\geq (\hat{r}_\beta(x)-\epsilon)n'_{k+1}\sum_{i=0}^{k-1}\left(\frac{\hat{r}_\beta(x)-\epsilon}{r_\beta(x)+\epsilon}\right)^i\\&= (\hat{r}_\beta(x)-\epsilon)n'_{k+1}\sum_{i=0}^\infty\left(\frac{\hat{r}_\beta(x)-\epsilon}{r_\beta(x)+\epsilon}\right)^i-n'\\&\geq n'_{k+1}\left(\frac{\hat{r}_\beta(x)r_\beta(x)}{r_\beta(x)-\hat{r}_\beta(x)}-\epsilon'\right).\end{aligned}$$
\end{proof}

We will prove Theorem \ref{t2} by dividing into three cases: $r=\hat{r}=0$,  $0\leq \frac{r}{1+r}<\hat{r}\leq +\infty$ and $0\leq \hat{r}\leq \frac{r}{1+r},\ 0<r\leq +\infty$.
\subsection{Case for  $r=\hat{r}=0$}
Note that $\hat{r}_{\beta}(x)\leq r_{\beta}(x)$ for all $x\in[0,1)$. Then $$\{x\in[0,1):r_\beta(x)=0\}\subseteq \{x\in[0,1):\hat{r}_\beta(x)=0\}.$$ So if we prove that  $r_\beta(x)=0$ for $\mathcal{L}$-almost all $x\in[0,1)$, we have $\hat{r}_\beta(x)=0$ for $\mathcal{L}$-almost all $x\in[0,1)$.
%We first give a result due to Philipp \cite{PH} to obtain the size of the set $\{x\in[0,1):r_\beta(x)=0\}$.
%\begin{theorem}[Philipp \cite{PH}]\label{ph}
%Let $\{B_n\}_{n=1}^\infty$ be an arbitrary sequence of intervals contained in $[0,1)$. For any positive integer $N$ and $x\in[0,1)$, denote by $A(N,x)$ the number of positive integer $n\leq N$ such that $T^n_\beta x\in B_n$. Put $$\phi(N)=\sum_{n\leq N} \nu_\beta(B_n).$$ Then
%$$A(N,x)=\phi(N)+O\left(\phi^{\frac{1}{2}}(N)\log^{\frac{3}{2}+\epsilon}\phi(N)\right),\ \epsilon>0$$for almost all $x\in [0,1)$ where the constant implied by $O$
%is an absolute constant.
%\end{theorem}

%Let $\varphi:\ \N \rightarrow \R^+$ be a function with $\varphi(n)\rightarrow 0$ as $n\rightarrow +\infty$. Let $B_n=B(x,\varphi(n))$ which is  an interval whose center is $x$ with radius $\varphi(n)$. By the equivalence between $\nu_\beta$ and $\mathcal{L}$ (see (\ref{nu})), the following corollary is straightforward.
%\begin{corollary}\label{phc}
%The set $$\{x\in[0,1):|T_\beta^nx-x|\leq \varphi(n){\rm\ for\ infinitely\ many\ }n\in \N\}$$ is of full Lebesgue measure if $\sum\limits_{n=1}^\infty\varphi(n)=+\infty$. Otherwise, it is of null Lebesgue measure.
%\end{corollary}
%\begin{lemma}\label{le}
%We have $\mathcal{L}\{x\in[0,1):r_\beta(x)=0\}=1$. Moreover, $\mathcal{L}\{x\in[0,1):\hat{r}_\beta(x)=0\}=\mathcal{L}(R_\beta(0,0))=1$.
%\end{lemma}
%\begin{proof}
Hence, we only need to prove $$\mathcal{L}\{x\in[0,1):r_\beta(x)>0\}=0.$$
Since $\sum\limits_{n=1}^\infty\beta^{-\frac{1}{k}n}<\infty$ for all $k\geq 1$, it follows from Philipp \cite{PH}(Theorem 2A) that
$$\mathcal{L}\{x\in[0,1):|T_\beta^nx-x|\leq \beta^{-\frac{1}{k}n} {\rm\ for\ infinitely\ many\ }n\in\N\}=0.$$ Moreover, we have $$\{x\in[0,1):r_\beta(x)>0\}=\bigcup_{k=1}^\infty\left\{x\in[0,1):r_\beta(x)>\frac{1}{k}\right\}$$ and $$\left\{x\in[0,1):r_\beta(x)>\frac{1}{k}\right\}\subseteq \{x\in[0,1):|T_\beta^nx-x|\leq \beta^{-\frac{1}{k}n} {\rm\ for\ infinitely\ many\ }n\in\N\}.$$ Thus, $$\mathcal{L}\{x\in[0,1):r_\beta(x)>0\}=0.$$
%The rest of the statement is immediate.
%\end{proof}

\subsection{Case for $0\leq \frac{r}{1+r}<\hat{r}\leq +\infty$}

When $\varepsilon(x,\beta)$ is periodic, we have $\hat{r}=r=+\infty$ and the set with such $\varepsilon(x,\beta)$ is countable.

When $\varepsilon(x,\beta)$ is not periodic, let $m_k$ and $n_k$ be defined in Lemma \ref{mn}. We will construct a countable set $D$ such that $\varepsilon(x,\beta)\in D$ for all $x\in R_\beta(\hat{r},r)$ with $0\leq \frac{r}{1+r}<\hat{r}\leq +\infty$. By Lemma \ref{com}, there is an integer $k'$, such that for all $k\geq k'$, we have $n_{k+1}<m_k$.  Let $t_k$ be defined by (\ref{tk}). Suppose $t_{k}=an_{k}+p_k$ where $1\leq a\leq \lfloor r\rfloor+1$ and $1\leq p_k<n_{k}$. Then, by (\ref{e1}),
\begin{equation}\label{e2}
\left(\varepsilon_1,\ldots,\varepsilon_{t_k}\right)=\left(\left(\varepsilon_1,\ldots,\varepsilon_{n_k}\right)^a,\varepsilon_1,\ldots,\varepsilon_{p_k}\right), \quad \varepsilon_{t_k+1}\neq \varepsilon_{p_k+1}.
\end{equation}where $\omega^k=(\underbrace{\omega,\ldots,\omega}_k)$ for any $\omega\in \Sigma^\ast_\beta,\ k\in \mathbb{N}$. Therefore, when $n_{k+1}\leq t_k\leq m_k+1$, by the definitions of $\{n_k\}_{k=1}^\infty$ and $\{t_k\}_{k=1}^\infty$, we have\begin{equation}\label{e3}
\left(\varepsilon_1,\ldots,\varepsilon_{t_{k}}\right)= \left(\varepsilon_1,\ldots, \varepsilon_{n_k},\varepsilon_1,\ldots,\varepsilon_{n_{{k+1}}-n_{k}},\varepsilon_1,\ldots,\varepsilon_{t_k-n_{k+1}}\right),\  \varepsilon_{t_{k}+1}=\varepsilon_{t_{k}-n_{k+1}+1}.
\end{equation} By comparing the equalities (\ref{e1}), (\ref{e2}) and  (\ref{e3}), we can check that $n_{k+1}\neq a'n_k$ for all integer $1\leq a'\leq a$. Suppose that there are two integers $1\leq a'\leq a$ and  $1\leq j_k<n_k$ such that $n_{{k+1}}=a'n_k+j_k$. Then $$\left(\varepsilon_1,\ldots,\varepsilon_{n_{k+1}}\right)=\left(\left(\varepsilon_1,\ldots,\varepsilon_{n_{k}}\right)^{a'},\varepsilon_1,\ldots, \varepsilon_{j_k}\right).$$ When $t_k<n_{k+1}<m_k$, by (\ref{1}), (\ref{2}) and  (\ref{e2}), it follows that $$\left(\varepsilon_1,\ldots,\varepsilon_{n_{k+1}}\right)=\left(\left(\varepsilon_1,\ldots,\varepsilon_{n_k}\right)^a,\varepsilon_1,\ldots, \varepsilon_{p_k},\varepsilon_{p_k+1}-1,\varepsilon_1^\ast(\beta),\ldots,\varepsilon_{n_{k+1}-p_k-1}^\ast(\beta)\right)$$ or $$\left(\varepsilon_1,\ldots,\varepsilon_{n_{k+1}}\right)=\left(\left(\varepsilon_1,\ldots,\varepsilon_{n_k}\right)^a,\varepsilon_1,\ldots, \varepsilon_{p_k},\varepsilon_{p_k+1}+1,0^{n_{k+1}-p_k-1}\right).$$

Now we will construct the countable set $D$. For all $\omega=(\omega_1,\ldots,\omega_n) \in\Sigma_\beta^n$, define $$s_1(\omega)=\min\left\{1\leq i<n:\omega_{i+1}=\omega_1\right\}.$$ If such $i$ does not exist, let $s_1(\omega)=n$. If $s_1(\omega)<n$, let
$$t_1(\omega)=\max\left\{s_1(\omega)< i\leq n:\left(\omega_{s_1(\omega)+1},\ldots,\omega_i\right)=\left(\omega_1,\ldots,\omega_{i-s_1(\omega)}\right)\right\}.$$ Suppose that $s_k(\omega)$ and  $t_k(\omega)$ have been defined, let $$s_{k+1}(\omega)=\min\{ s_k(\omega)\leq i<n:\omega_{i+1}=\omega_1\}.$$ If this $i$ does not exist, let $s_{k+1}(\omega)=n$. If $s_{k+1}(\omega)<n$,  let  $$t_{k+1}(\omega)=\max\left\{s_k(\omega)<i\leq n:\left(\omega_{s_k(\omega)+1},\ldots,\omega_i\right)=\left(\omega_1,\ldots,\omega_{i-s_k(\omega)}\right)\right\}.$$  Let $k(\omega)=\max\{k:1\leq s_k(\omega)<n\}$ and $k(\omega)=0$ if such $k$ does not exist.

For all $\omega=(\omega_1,\ldots,\omega_n) \in\Sigma_\beta^n$,
when $k(\omega)>0$, let $$M'(\omega)=\bigcup_{a=1}^{\lfloor r\rfloor+1}\bigcup_{k=1}^{k(\omega)}\bigcup_{j=0}^{\lfloor r\rfloor n}\left\{\left(\omega^a,u,\omega_{t_k(\omega)+1}-1,\varepsilon_1^\ast(\beta),\ldots,\varepsilon_j^\ast(\beta)\right) :u=\left(\omega_1,\ldots,\omega_{t_k(\omega)}\right),\ \omega_{t_k(\omega)+1}>0\right\}$$ where $\left(\omega^a,u,\omega_{t_k(\omega)+1}-1,\varepsilon_1^\ast(\beta),\ldots,\varepsilon_j^\ast(\beta)\right) =\left(\omega^a,u,\omega_{t_k(\omega)+1}-1\right)$ when $j=0$, and $$M''(\omega)=\bigcup_{a=1}^{\lfloor r\rfloor+1}\bigcup_{k=1}^{k(\omega)}\bigcup_{j=0}^{\lfloor r\rfloor n}\left\{\left(\omega^a,u,\omega_{t_k(\omega)+1}+1,0^j\right) :u=\left(\omega_1,\ldots,\omega_{t_k(\omega)}\right),\ 0\leq \omega_{t_k(\omega)+1}<\lfloor\beta \rfloor\right\}$$ where $\left(\omega^a,u,\omega_{t_k(\omega)+1}+1,0^j\right)=\left(\omega^a,u,\omega_{t_k(\omega)+1}+1\right)$ when $j=0$.  When $k(\omega)=0$, let $$M'(\omega)=\bigcup_{a=1}^{\lfloor r\rfloor+1}\bigcup_{j=0}^{\lfloor r\rfloor n} \left\{\left(\omega^a,\omega_1-1,\varepsilon_1^\ast(\beta),\ldots,\varepsilon_j^\ast(\beta)\right), \omega_1>0\right\}$$ where $\left(\omega^a,\omega_1-1,\varepsilon_1^\ast(\beta),\ldots,\varepsilon_j^\ast(\beta)\right)=\left(\omega^a,\omega_1-1\right)$ when $j=0$, and $$M''(\omega)=\bigcup_{a=1}^{\lfloor r\rfloor+1}\bigcup_{j=0}^{\lfloor r\rfloor n}\left\{\left(\omega^a,\omega_1+1,0^j\right), 0\leq \omega_1<\lfloor \beta\rfloor\right\}$$ where $\left(\omega^a,\omega_1+1,0^j\right)=\left(\omega^a,\omega_1+1\right)$ when $j=0$.

Now fix $\omega=(\omega_1,\ldots,\omega_n) \in\Sigma_\beta^n$ and let $M(\omega)=M'(\omega)\bigcup M''(\omega)\bigcup\left\{\omega\right\}$. Set $M_1(\omega)=M(\omega)\bigcup \{(\omega^2)\}$. Suppose that $M_k(\omega)$ has been defined. Let
$$M_{k+1}(\omega)=\bigcup_{v_k\in M_k(\omega)}M(v_k)\bigcup\left\{(\omega^k)\right\}.$$ Then we have $\sharp M_k(\omega)<\infty$ for all $k\geq 1$ and we also have  $M_k(\omega)\subset \A^\ast$. Let $$D=\bigcup_{n=1}^\infty\bigcup_{\omega\in\Sigma_\beta^n} \bigcup_{k=1}^\infty M_k(\omega).$$ By the former analysis in this section, for any $x\in R_\beta(\hat{r},r)$ with $0\leq \frac{r}{1+r}<\hat{r}\leq +\infty$, we have $\varepsilon(x,\beta)\in D$ and the set $D$ is countable. As a consequence, the set  $R_\beta(\hat{r},r)\left(0\leq \frac{r}{1+r}<\hat{r}\leq +\infty\right)$ is countable.
%\begin{remark}
%In our proof, when $s_k(\omega)=n$, we can stop choosing $s_{k+1}(\omega)$ since $s_{k+1}(\omega)$ will do not exists. Moreover, $k(\omega)=0$ when $s_1(\omega)=n$ and in this case, we can see that $\omega$ does not return to itself.
%\end{remark}
\subsection{Case for $0\leq \hat{r}\leq \frac{r}{1+r},\ 0<r\leq +\infty$}
Note that $\{x\in[0,1):r_\beta(x)=+\infty\}\subseteq \{x\in[0,1):r_\beta(x)\geq r\}$ for all $0\leq r\leq+\infty$. When $0\leq \hat{r}\leq \frac{r}{1+r},\ r=+\infty$,
%note that $R_\beta(r,\hat{r})\subseteq \{x\in[0,1): r_\beta(x)=+\infty\}$, the Huasdorff dimension of $R_\beta(r,\hat{r})$ is obtained by the following %result.
%\begin{lemma}\label{le1}
%Let $\beta>1$ be a real number. Then $\dim_{\rm H}\{x\in[0,1): r_\beta(x)=+\infty\}=0.$
%\end{lemma}
%\begin{proof}
 by (\ref{tw}), it holds that $$\dim_{\rm H}\{x\in[0,1):r_{\beta}(x)\geq r\}=\frac{1}{1+r}.$$ Consequently, letting $r \rightarrow +\infty$, we conclude that $$\dim_{\rm H}\{x\in[0,1): r_\beta(x)=+\infty\}=0.$$ By the fact that $R_\beta(r,\hat{r})\subseteq \{x\in[0,1): r_\beta(x)=+\infty\}$, we have $\dim R_\beta(r,\hat{r})=0.$
%\end{proof}

When $0\leq \hat{r}\leq \frac{r}{1+r},\ 0<r<+\infty$. Our proof is divided into two parts.

\subsubsection{The upper bound of $\dim_{\rm H}R_\beta(\hat{r},r)$}
We now construct a covering of the set $R_\beta(\hat{r},r)$ with $0\leq \hat{r}\leq \frac{r}{1+r},\ 0<r<+\infty$. Let $\{n'_k\}_{k=1}^\infty$ and $\{m'_k\}_{k=1}^\infty$ be the sequences such that
\begin{equation}\label{lim}
\lim_{k\rightarrow \infty}\frac{m'_k-n'_k}{n'_k}= r\quad {\rm and}\quad \lim_{k\rightarrow \infty}\frac{m'_k-n'_k}{n'_{k+1}}\geq \hat{r}.
\end{equation}

Given $k\geq 1$, we collect all of the points $x$ with $\hat{r}_\beta(x)\geq \hat{r}$ and $r_\beta(x)=r$. We first calculate the possible choices of digits among the $m'_k$ prefix of $(\omega_1,\omega_2,\ldots)$. For the ``free'' blocks in the prefix of length $m'_k$ of the infinite sequence $\varepsilon(x,\beta)$. Write their lengths as $d_1,\cdots,d_k$. It follows immediately that $d_i=0$ when $n'_i<m'_{i-1}$ and  $d_i=n'_i-m'_{i-1}$ when $n'_i>m'_{i-1}$ for all $2\leq i\leq k$. Moreover, $d_1=n_1$. Let $m_0=0$. By (\ref{i3}) and the choice of $m_k$ and $n_k$, it follows from Lemma \ref{rep} that, for any sufficiently large $k\geq \max\{k_0,k_1,k'\}$,
\begin{equation}\label{i4}
\sum_{i=1}^k d_i\leq \sum_{i=1}^{k}(n'_i-m'_{i-1})+\sum_{n'_i<m'_{i-1}, k_0\leq i\leq k}(m'_{i-1}-n'_i).
\end{equation}
Applying Lemma \ref{rep}, we have
$$\sum_{n'_i<m'_{i-1, 1\leq i\leq k}}(m'_{i-1}-n'_i)\leq \delta\sum_{n'_i<m'_{i-1}, k_0\leq i\leq k} m'_{i-1}+n_0$$ where $n_0=\delta\sum\limits_{n'_i<m'_{i-1}, k_0\leq i\leq k} (m'_{i-1}-n'_i)<+\infty$. By (\ref{i1}) and Lemma \ref{rep}, we have
\begin{equation}\label{i5}
\begin{aligned}
\sum_{n'_i<m'_{i-1, 1\leq i\leq k}}(m'_{i-1}-n'_i)&\leq\delta\sum_{n'_i<m'_{i-1}, k_0\leq i\leq k}(1+r+\epsilon)n_{i-1}+n_0\\&\leq \delta\left(1+\frac{\hat{r}}{2+\hat{r}}+\ldots+\left(\frac{\hat{r}}{2+\hat{r}}\right)^k\right)(1+r+\epsilon)n_k+n_0.
\end{aligned}
\end{equation}By (\ref{i3}), it holds that
\begin{equation}\label{i6}
\sum_{i=1}^{k}(n'_i-m'_{i-1})=n'_k-\sum_{i=1}^{k-1}(m'_i-n'_i)\leq n'_k\left(1-\frac{\hat{r}r}{r-\hat{r}} +\epsilon'\right).
\end{equation} Combining (\ref{i4}), (\ref{i5}) and (\ref{i6}), we obtain $$\sum_{i=1}^k d_i\leq n'_k\left(1-\frac{\hat{r}r}{r-\hat{r}} +\epsilon''\right),$$
where $\epsilon''$ is a small enough real number. By Theorem \ref{R}, we deduce that, for every blocks with length $d_i$, there are no more than $$\frac{\beta}{\beta-1}\beta^{d_i}$$ ways of the words can be chosen. Thus, there are at most
$$\left(\frac{\beta}{\beta-1}\right)^k\beta^{\sum\limits_{i=1}^k d_i}\leq \left(\frac{\beta}{\beta-1}\right)^k\beta^{n'_k\left(1-\frac{\hat{r}r}{r-\hat{r}} +\epsilon''\right)}$$ choices of the ``free'' blocks in total. Notice that there are at most $n_k$ possible choices for the first index of the $k$ blocks. This indicates that there are at most ${n'_k}^k$ possible choices for the position of the ``free'' blocks. For the ``fixed'' block, it follows from (\ref{e4}), (\ref{1}), (\ref{2}) that the block $(\varepsilon_1,\ldots,\varepsilon_{m'_k-n'_k})$ has at most $2\lfloor\beta\rfloor(m'_k-n'_k+1)\leq 2\beta(m'_k-n'_k+1)$ choices which means there are at most $\left(2\beta(m'_k-n'_k+1)\right)^k$ choices of the ``fixed'' blocks in total. By Lemma \ref{rep} and (\ref{lim}), for all sufficient large $k$, the set of all real number belonging to $R_\beta(\hat{r},r)$ is contained in a union of no more than $$ \left(\frac{2(m'_k-n'_k)n'_k\beta^2}{\beta-1}\right)^k \beta^{n'_k\left(1-\frac{\hat{r}r}{r-\hat{r}} +\epsilon''\right)} \leq\left(\frac{2(r+\epsilon''){n'_k}^2\beta^2}{\beta-1}\right)^{C\log n'_k} \beta^{n'_k\left(1-\frac{\hat{r}r}{r-\hat{r}} +\epsilon''\right)}$$ cylinders of order $m'_k$ whose length is at most $$\beta^{-m'_k}\leq \beta^{(1+r-\epsilon'')n'_k},$$ where the last inequalities follows from (\ref{lim}). Denote $$s_0=\frac{r-(1+r)\hat{r}+\epsilon''(r-\hat{r})}{(r-\hat{r})(r+1-\epsilon'')}.$$ Then for any $s>s_0$, we have
$$\H^s\left(R_\beta(\hat{r},r)\right)\leq \sum_{n=1}^\infty\beta \left(\frac{2(r+\epsilon'')n^2\beta^2}{\beta-1}\right)^{C\log n}\beta^{-(1+r-\epsilon'')ns+n\left(1-\frac{\hat{r}r}{r-\hat{r}} +\epsilon''\right)}<+\infty.$$

Letting $\epsilon''\rightarrow 0$, we conclude that $$\dim_{\rm H} R_\beta(\hat{r},r)\leq \frac{r-(1+r)\hat{r}}{(r-\hat{r})(r+1)}.$$

\subsubsection{Construction of Cantor Set}

We construct a Cantor subset of $R_\beta(\hat{r},r)\left(0\leq \hat{r}\leq \frac{r}{1+r},\ 0<r<+\infty\right)$ as follows. Fix $\delta>0$. Let $\beta_N$ be the approximation of $\beta$ which is defined in Section 2. Notice that $\beta_N\rightarrow \beta$ as $N\rightarrow \infty$, we can choose sufficiently large integer $N$  with $\varepsilon_N^\ast(\beta)>0$ and $M$ large enough such that
\begin{equation}\label{m}
\frac{\sharp \Sigma_{\beta_N}^M}{M}-M-1\geq \frac{\beta_N^M}{M}-M-1\geq \beta^{M(1-\delta)}.
\end{equation}

Now we choose two sequences $\{n_k\}_{k=1}^\infty$ and  $\{m_k\}_{k=1}^\infty$ such that $n_k<m_k<n_{k+1}$ with $n_1>2M$. The sequence $\{m_k-n_k\}_{k=1}^\infty$ is non-decreasing. In addition,
\begin{equation}\label{inf2}
\lim_{k\rightarrow \infty}\frac{m_k-n_k}{n_{k+1}}=\hat{r}
\end{equation}
and
\begin{equation}\label{sup2}
\lim_{k\rightarrow \infty}\frac{m_k-n_k}{n_k}=r.
\end{equation}
Actually, we can choose the following sequences.

(1) When $\hat{r}=0$, $0<r<+\infty$, let
$$n'_k=k^k\quad {\rm and}\quad  m'_k=\left\lfloor\left(r+1\right) k^k\right\rfloor.$$  By a small adjustment, we can obtain the required sequences.

(2) When $0< \hat{r}\leq \frac{r}{1+r},\ 0<r<+\infty$, let
$$n'_k=\left\lfloor\left(\frac{r}{\hat{r}}\right)^k\right\rfloor\ {\rm and}\  m'_k=\left\lfloor(r+1)\left(\frac{r}{\hat{r}}\right)^k\right\rfloor.$$ Note that $r<\hat{r}$. Both of the sequences $\{n'_k\}_{k=1}^\infty$ and $\{m'_k\}_{k=1}^\infty$ increase to infinity as $k$ increases. We can adjust these sequences to make sure that they satisfy the required properties. Without any ambiguity, the statement that $\omega\in\Sigma_{\beta_N}^\ast$ is full means that $\omega$ is full when regarding it as an element of $\Sigma_{\beta}^\ast$.

For all $k\in\N$, let $m_k=\ell_kn_k+p_k$ with $0\leq p_k<n_k$ and  $n_{k+1}-m_k=t_k M+q_k$ with $0\leq q_k<M$.  For any integer $n>N$ and $\omega\in\Sigma_{\beta_N}^n$, define the function $a_{p_k}$ of $\omega$ by
$$a_{p_k}(\omega)=\left\{
\begin{aligned}
0^{p_k} & , &{\rm when}\ \ p_k \leq N, \\
\omega|_{p_k-N},0^N & , &{\rm when}\ \ p_k > N,\\
\end{aligned}
\right.$$  where $\omega|_i=(\omega_1,\ldots,\omega_i)$. Then $a_{p_k}(\omega)$ is full when regarding it as an element of $\Sigma_\beta^\ast$. For convenience, denote $$\omega^i=(\omega_{n-i+1},\ldots,\omega_n,\omega_{i+1},\ldots,\omega_{n-i})$$ for all $1\leq i<n$ and $\omega^n=\omega$. Let $$\D_1=\left\{v_1=\left((u)^{\lfloor\frac{n_1}{M}\rfloor},0^{n_1-\lfloor\frac{n_1}{M}\rfloor M}\right):u\in\Sigma_{\beta_N}^M  {\rm\ is\ full\ and\ } u\neq 0^M\right\}.$$ Set $$\G_1=\left\{u_1=\left((v_1)^{\ell_1},a_{p_1}(v_1)\right):v_1\in \D_1\right\}.$$ Fix $v_1=\left((u)^{\lfloor\frac{n_1}{M}\rfloor},0^{n-\lfloor\frac{n_1}{M}\rfloor M}\right)\in \D_1$.
Define
$$\M=\left\{\omega \in \Sigma_{\beta_N}^M: \omega {\rm\ is\ full\  and\ } \omega \neq u^i,\ \forall\ 1\leq i\leq n\right\}.$$ Suppose that $\D_{k-1}$ and $\G_{k-1}$ are well defined. Fix $u_{k-1}\in \G_{k-1}$.  Let $$\D_k=\left\{v_k=\left(u_{k-1}^0,\ldots,u_{k-1}^{t_{k-1}-1},0^{q_{k-1}}\right):u_{k-1}^i \in \M, \forall\ 0\leq i\leq t_{k-1}-1\right\}.$$
Define $$\G_k=\left\{u_k=\left((u_{k-1},v_k)^{\ell_k},a_{p_{k}}(u_{k-1},v_k)\right):u_{k-1}\in \G_{k-1}, v_k\in\D_k\right\}.$$ We can see that all of the words in $\M$ and $\D_k$ are full when regarding them as an element of $\Sigma_\beta^\ast$.  Hence, by Theorem \ref{AB}, $\G_k$ is well defined.

Now define $$E_N=\bigcap_{k=1}^{\infty}\bigcup_{u_k \in \G_k}I_{m_k}(u_k).$$ We claim that $E_N$ is a subset  $R_\beta(\hat{r},r)$.

In fact, for any $x\in E_N$, suppose $\varepsilon(x,\beta)=(\varepsilon_1,\varepsilon_2,\ldots)$. We first prove that for every $n$ with $m_k\leq n<n_{k+1}$, we have $$(\varepsilon_{n+1},\ldots,\varepsilon_{n+2M})\neq (\varepsilon_1,\ldots,\varepsilon_{2M})=(\varepsilon_1,\ldots,\varepsilon_M,\varepsilon_1,\ldots,\varepsilon_M),$$ where the last equality follows from the fact that $n_1>2M$  and the construction of $E_N$. In fact, if it is not true, then there is $m_k\leq n<n_{k+1}-2M$ such that $$(\varepsilon_{n+1},\ldots,\varepsilon_{n+2M})= (\varepsilon_1,\ldots,\varepsilon_{2M}).$$ Assume $n=m_k+tM+q\ (0\leq t\leq t_k-1)$. Then we have $$\left(\varepsilon_{n+1},\ldots,\varepsilon_{n+2M}\right)=\left(\varepsilon_{m_k+tM+q+1},\ldots,\varepsilon_{m_k+(t+1)M+1},\ldots, \varepsilon_{m_k+(t+2)M},\ldots, \varepsilon_{m_k+(t+2)M+q}\right).$$ This implies $$\left(\varepsilon_{m_k+(t+1)M+1},\ldots,\varepsilon_{m_k+(t+2)M}\right)=\left(\varepsilon_{M-q+1},\ldots,\varepsilon_M,\varepsilon_1,\ldots,\varepsilon_q\right)\notin \M,$$ which is a contradiction. Note that $n\geq 2M$. Suppose $$(\varepsilon_{n+1},\ldots,\varepsilon_{n+j})=(\varepsilon_1,\ldots,\varepsilon_j),\  \varepsilon_{n+j+1}\neq \varepsilon_{j+1}.$$ Without loss of generality, suppose $\varepsilon_{n+j+1}> \varepsilon_{j}+1$. Notice that both $(\varepsilon_{n+1},\ldots,\varepsilon_{n+j+1})$ and $(\varepsilon_1,\ldots,\varepsilon_{j+1})$ belong to $\Sigma_{\beta_N}^M$. Then
\begin{equation}\label{3}
\begin{aligned}
&|T_\beta^n x-x|=T_\beta^n x-x\\&\geq \frac{\varepsilon_{n+1}}{\beta}+\cdots+\frac{\varepsilon_{n+j}}{\beta^j}+\frac{\varepsilon_{n+j+1}}{\beta^{j+1}}-\left( \frac{\varepsilon_{n+1}}{\beta}+\cdots+\frac{\varepsilon_{n+j}}{\beta^j}+\frac{\varepsilon_j}{\beta^{j+1}}+\frac{\varepsilon_1^\ast(\beta)}{\beta^{j+2}}+\cdots+ \frac{\varepsilon_N^\ast(\beta)-1}{\beta^{j+N+1}} \right)\\&\geq \beta^{-j+N+1}.
\end{aligned}
\end{equation} Hence, for any $m_k\leq n<n_{k+1}$, we have
$$|T_\beta^nx-x| \geq \beta^{-(2M+N+1)}.$$ We now show that $r_\beta(x)=r.$

On the one hand, for any $\delta>0$, notice that $$\lim_{k\rightarrow\infty}\frac{m_k-n_k+N}{n_k}=r.$$ Then there exits $k_0$ large enough such that for all $k\geq k_0$, we have $m_k-n_k+N<(r+\delta)n_k$. Consequently, by (\ref{in}), for all $n\geq n_{k_0}$, there is $k\geq k_0$, such that $n_k\leq n< n_{k+1}$. The same argument as (\ref{3}) gives $$|T_\beta^n x-x|\geq \beta^{-(m_k-n_k+N+1)}>\beta^{-(r+\delta)n_k}\geq\beta^{-(r+\delta)n}.$$ By the definition of $r_\beta(x)$, it holds that $r_\beta(x)<r+\delta$ for all $\delta>0$. So $r_\beta(x)\leq r.$

On the other hand, for all $\delta>0$, there is $k_0$ such that for any $k\geq k_0$. Thus, we have $m_k-n_k\geq (r-\delta)n_k$. When $n=n_k$, we obtain $$|T^{n_k}_\beta x-x|<\beta^{-(m_k-n_k)}\leq\beta^{-(r-\delta)n_k}.$$ As a consequence, $r_\beta(x)\geq r-\delta$ for any $\delta>0$, which implies $r_\beta(x)\geq r.$

The proof of $\hat{r}_\beta(x)=\hat{r}$ is similar to the argument of $r_\beta(x)=r.$ We leave the details to the readers.

%On the one side, for all $\delta>0$, since $$\lim_{k\rightarrow\infty}\frac{m_k-n_k}{n_{k+1}}=\hat{r}.$$ There exits sufficient large $k_0$ such that for all $k\geq k_0$, we have $m_k-n_k\geq (\hat{r}-\delta)n_{k+1}$. Then for each $n\geq n_{k_0}$, there is $k\geq k_0$, such that $n_k\leq n< n_{k+1}$ and $$|T_\beta^{n_k} x-x|< \beta^{-(m_k-n_k)}\leq\beta^{-(\hat{r}-\delta)n_{k+1}}\leq\beta^{-(\hat{r}-\delta)n}.$$ The definition of $\hat{r}_\beta(x)$ yields that $\hat{r}_\beta(x)\geq \hat{r}-\delta$ for all $\delta>0$. Thus, $\hat{r}_\beta(x)\geq \hat{r}.$
%
%On the other side, for every $\delta>0$, there is an integer $k_0$ such that for any $k\geq k_0$, we have $m_k-n_k+N\leq(\hat{r}+\delta)(n_{k+1}-1)$. Then for all $n\geq 1$, there is an integer $k\geq k_0$, such that $1\leq n\leq n_{k+1}-1$. By the same argument as (\ref{3}), it holds that for any $1\leq n\leq n_{k+1}-1$,  $$|T^n_\beta x-x|\geq \beta^{-(m_k-n_k+N+1)}\geq\beta^{-(\hat{r}+\delta)(n_{k+1}-1)}.$$ So $\hat{r}_\beta(x)<\hat{r}+\delta$ for any $\delta>0$, which means $\hat{r}_\beta(x)\leq \hat{r}.$

%\begin{proof} For any $n\geq1$, there is an integer $k\in \N$ such that $n_k< n\leq n_{k+1}$. By the construction of $E_N$, we have
%$$\beta^{-m_k+n_k+N}\leq|T_\beta^Nx-x|\leq \beta^{-m_k+n_k}.$$
%Then, $$r_\beta(x)=\lim_{k\rightarrow\infty}\frac{m_k-n_k}{n_k}=r$$ and $$\hat{r}_\beta(x)=\lim_{k\rightarrow\infty}\frac{m_k-n_k}{n_{k+1}}=\hat{r}.$$
Our Cantor set is therefore constructed.
\subsubsection{ The lower bound of $\dim_{\rm H}R_\beta(\hat{r},r)$}
The rest of this section is devoted to estimating the lower bound of Hausdorff dimension of $E_N$ by the modified mass distribution principal.

As the classical method of giving the lower bound of $\dim_{\rm H} E_N$, we first define a Borel probability measure $\mu$ on $E_N$. Set $$\mu([0,1))=1,\ {\rm and}\ \mu\left(I_{m_1}(u_1\right))=\frac{1}{\sharp \G_1}=\frac{1}{\sharp \D_1},\ {\rm for}\ u_1 \in \G_1.$$ For each $k\geq 1,$ if $u_{k+1}=\left((u_k,v_{k+1})^{\ell_{k+1}},a_{p_{k+1}}(u_k,v_{k+1})\right)\in \G_{k+1}$, let
\begin{equation}\label{mu}
\mu\left(I_{m_{k+1}}(u_{k+1})\right)=\frac{\mu\left(I_{m_k}(u_k)\right)}{\sharp \D_{k+1}}.
\end{equation}  If $u_{k+1} \notin \G_{k+1}$ ($k\geq 1$), let $\mu\left(I_{m_{k+1}}(u_{k+1})\right)=0$. For each $n\in\N$ and each cylinder $I_n$ with $I_n\cap E_N\neq \emptyset$, let $k\geq1$ be the integer such that $m_k<n\leq m_{k+1}$, and set $$\mu(I_n)=\sum_{I_{m_{k+1}}\subseteq I_n}\mu\left(I_{m_{k+1}}\right),$$ where the sum is taken over all the basic cylinders associate to $u_{k+1}\in\G_{k+1}$ contained in $I_n$. We can see that $\mu$ satisfies the consistency property which ensures that it can be uniquely extended to a Borel probability measure on $E_N$.

Now we will estimate the local dimension $\liminf\limits_{n\rightarrow \infty}\frac{\log \mu(I_n)}{\log|I_n|}$ for all basic cylinder $I_n$ with $I_n\cap E_N\neq \emptyset$. We claim that, there exists $k_0$, for all $k\geq k_0$, we have
\begin{equation}\label{qk}
\sharp\G_k\geq c^k\beta^{(1-\delta)\sum\limits_{i=1}^{k-1}(n_{i+1}-m_i)}
\end{equation}
for all $k\geq k_0$. Combining Theorems \ref{R}, \ref{no} and the definition of $\M$, by (\ref{m}), we have $$\sharp \M\geq  \frac{\sharp\Sigma_\beta^M}{M}-M \geq \beta^{M(1-\delta)}.$$

Immediately, for any $k\geq1$, the relationship between $\M$ and $\D_k$ gives
$$\sharp\D_k=(\sharp \M)^{t_{k-1}}\geq \beta^{Mt_{k-1}(1-\delta)}= \beta^{-M(1-\delta)}\cdot\beta^{M(1-\delta)(t_{k-1}+1)}\geq c\beta^{(1-\delta)(n_k-m_{k-1})},$$ where $c=\beta^{-M(1-\delta)}$.
As a consequence,
$$\sharp \G_k=\prod_{i=1}^{k}\sharp \D_i\geq c^k\beta^{(1-\delta)\sum\limits_{i=1}^{k-1}(n_{i+1}-m_i)}.$$

Therefore, for all $i\geq2$,
\begin{equation}\label{m1}
\mu(I_{m_i})=\frac{1}{\sharp\G_i}\leq \frac{1}{c^k\beta^{(1-\delta)\sum\limits_{j=1}^{i-1}(n_{j+1}-m_j)}}.
\end{equation} In gerneral $n\geq 1$, there is an integer $k\geq 1$ such that $m_k<n\leq m_{k+1}$. We distinguish two cases to estimate $\mu(I_n)$.

Case 1. When $m_k<n\leq n_{k+1}$, write $n=m_k+t M+q$ with $0\leq t\leq t_k$ and $0\leq q<M$. then
$$\mu(I_n)\leq \mu(I_{m_k})\cdot\frac{1}{(\sharp \M)^t}\leq  c^{-k}\beta^{-(1-\delta)\sum\limits_{j=1}^{k-1}(n_{j+1}-m_j)}\cdot\frac{1}{\beta^{M(1-\delta)t}}.$$

Case 2. When $n_{k+1}<n\leq m_{k+1}$, by the construction of $E_N$, it follows that $$\mu(I_n)\leq \mu(I_{n_{k+1}})=\mu(I_{m_{k+1}})\leq c^{-(k+1)} \beta^{-(1-\delta)\sum\limits_{j=1}^k(n_{j+1}-m_j)}.$$

By (\ref{in}), in both two cases, $$|I_n|\geq \frac{1}{\beta^{n+N}}.$$ Since $$\frac{a+x}{b+x}\geq\frac{a}{b}\ {\rm for\ all}\ 0<a\leq b,\ x\geq 0,$$ it holds that$$\liminf_{n\rightarrow\infty}\frac{\log \mu(I_n)}{\log|I_n|}\geq (1-\delta)\lim_{k\rightarrow\infty}\frac{\sum\limits_{j=1}^{k-1}(n_{j+1}-m_j)}{m_k}.$$ By (\ref{inf2}) and (\ref{sup2}), $$\lim_{k\rightarrow\infty}\frac{n_k}{m_k}=\frac{1}{1+r},\ \lim_{k\rightarrow\infty}\frac{n_{k}}{m_{k-1}}=\frac{r}{\hat{r}(1+r)}\ {\rm and}\ \lim_{k\rightarrow\infty}\frac{m_k}{m_{k-1}}=\frac{r}{\hat{r}}.$$ By the Stolz-Ces\`{a}ro theorem, we deduce that$$
\lim_{k\rightarrow\infty}\frac{\sum\limits_{j=1}^{k-1}(n_{j+1}-m_j)}{m_k}=\lim_{k\rightarrow\infty}\frac{n_k-m_{k-1}}{m_k-m_{k-1}}
=\lim_{k\rightarrow\infty}\frac{\frac{n_k}{m_{k-1}}-1}{\frac{m_k}{m_{k-1}}-1}=\frac{r-(1+r)\hat{r}}{(1+r)(r-\hat{r})}.$$
Hence, $$\liminf_{n\rightarrow \infty}\frac{\log \mu(I_n)}{\log|I_n|}\geq (1-\delta)\cdot\frac{r-(1+r)\hat{r}}{(1+r)(r-\hat{r})}.$$

Let $\delta\rightarrow0$. The modified mass distribution principle (Theorem \ref{mp}) gives $$\dim_{\rm H}E_N\geq \frac{r-(1+r)\hat{r}}{(1+r)(r-\hat{r})}.$$ for all sufficient large enough $N$.

%\textbf{Proof of Theorem \ref{t2}} It follows from Lemmas \ref{up} and \ref{lower} that, for all $0\leq \hat{r}\leq \frac{r}{1+r},\ 0<r<+\infty$, $$\dim_{\rm H}R_\beta(\hat{r},r)=\frac{r-(1+r)\hat{r}}{(1+r)(r-\hat{r})}.$$ Combing Lemmas \ref{le}, \ref{count} and \ref{le1}, we complete our proof. $\hfill\Box$

\section{Proof of Theorem \ref{t1}}
When $\hat{r}=0$, note that $$R_\beta(0,0)\subseteq\{x \in [0,1):\hat{r}_\beta(x)=0\}.$$  Theorem \ref{t2} yields that the set $\{x \in [0,1):\hat{r}_\beta(x)=0\}$ is of full Lebesgue measure. When $\hat{r}>1$, we have $$\hat{r}>\frac{r}{1+r}\geq0$$ for all $0\leq r\leq +\infty$. It follows from Theorem \ref{t2} that the set  $\{x \in [0,1):\hat{r}_\beta(x)\geq \hat{r}\}$ is countable. When $0<\hat{r}\leq 1$, applying the same process as Section 3.3.1, we deduce that, for all $\theta>0$,
\begin{equation}\label{max}
\begin{aligned}
\dim_{\rm H}\{x \in [0,1):\hat{r}_\beta(x)\geq \hat{r},\ r\leq r_\beta(x)\leq r+\theta\}&\leq\frac{r+\theta-(1+r+\theta)\hat{r}}{(1+r)(r+\theta-\hat{r})}\\& \leq \frac{r-(1+r)\hat{r}}{(1+r)(r-\hat{r})} +\frac{\theta(1-\hat{r})}{(1+r)(r-\hat{r})} \\& \leq \frac{r-(1+r)\hat{r}}{(1+r)(r-\hat{r})} +\frac{\theta(1-\hat{r})^2}{\hat{r}^2}.
\end{aligned}
\end{equation} Now considering the final part of (\ref{max}) as a function of $r$, we can find that $$\frac{r-(1+r)\hat{r}}{(1+r)(r-\hat{r})}+\frac{\theta(1-\hat{r})^2}{\hat{r}^2}\leq \frac{(1-\hat{r})^2}{(1+\hat{r})^2}+\frac{\theta(1-\hat{r})^2}{\hat{r}^2}$$ for all $r\geq0$.  Note that $$\{x \in [0,1):\hat{r}_\beta(x)\geq \hat{r}\}\subseteq \bigcup_{i=0}^{+\infty} \bigcup_{j=1}^n\left\{x \in [0,1):\hat{r}_\beta(x)\geq \hat{r},\ i+\frac{j-1}{n}\leq r_\beta(x)\leq i+\frac{j}{n} \right\}.$$ By the $\sigma$-stability of Hausdorff dimension (see \cite{FE}), we have $$\dim_{\rm H}\{x \in [0,1):\hat{r}_\beta(x)\geq \hat{r}\}\leq \left(\frac{1-\hat{r}}{1+\hat{r}}\right)^2+\frac{1}{n}\frac{(1-\hat{r})^2}{\hat{r}^2}.$$ Letting $n\rightarrow +\infty$, we therefore conclude that  $$\dim_{\rm H}\{x \in [0,1):\hat{r}_\beta(x)\geq \hat{r}\}\leq \left(\frac{1-\hat{r}}{1+\hat{r}}\right)^2.$$

Finally, we use the maximization method of Bugeaud and Liao \cite{BL} for the estimation of the lower bound of $\dim_{\rm{H}} \{x \in [0,1):\hat{r}_\beta(x)= \hat{r}\}$.  In fact, fix $0< \hat{r}\leq 1$, the function $r \mapsto \dim_{\rm H}R_\beta(\hat{r},r)$ is continuous and reaches its maximum at the unique point $r=\frac{2\hat{r}}{1-\hat{r}}$ . By calculation, we conclude that the maximum is exactly equal to $\left(\frac{1-\hat{r}}{1+\hat{r}}\right)^2$. Therefore, $$\dim_{\rm H}\{x \in [0,1):\hat{r}_\beta(x)= \hat{r}\}\geq \left(\frac{1-\hat{r}}{1+\hat{r}}\right)^2.$$

{\bf Acknowledgement}
The authors thank Lingmin Liao for his valuable suggestion.  This work was partially supported by NSFC 11771153, Guangzhou Natural Science Foundation 2018B0303110005.

\end{document}